\documentclass[12pt]{article}

\usepackage{amsfonts,amssymb,amsmath,amscd,amsthm}
\usepackage{graphicx}
\usepackage{xcolor}
\usepackage{geometry}
\usepackage{enumerate}
\usepackage{fancyhdr}
\usepackage{soul}

\usepackage{cite}

\usepackage{hyperref}
\hypersetup{
colorlinks   = true,
citecolor    = black,
linkcolor =black,
urlcolor=black,
}

\hypersetup{hidelinks}

\usepackage{authblk}

\newcommand{\email}[1]{%
    \normalsize\href{mailto:#1}{\color{black}{#1} }}

\makeatletter
\newcommand{\subjclass}[2][2020]{%
  \let\@oldtitle\@title%
  \gdef\@title{\@oldtitle\footnotetext{#1 \emph{Mathematics subject classification}: #2}}%
}
\newcommand{\keywords}[1]{%
  \let\@@oldtitle\@title%
  \gdef\@title{\@@oldtitle\footnotetext{\emph{Keywords}: #1.}}%
}
\makeatother

\usepackage{enumitem}  
\usepackage{calc}
\setlist{labelindent=1pt,itemsep=.5em}
\setlist[itemize]{leftmargin=1.2cm}
\setlist[enumerate]{itemindent=0em,leftmargin=1.2cm}
\setlist[enumerate,1]{label={\upshape(\roman*)}}

\allowdisplaybreaks

\def\figurename{Figure} 
\makeatletter
\renewcommand{\fnum@figure}[1]{\figurename~\thefigure.}
\makeatother

\def\tablename{Table} 
\makeatletter
\renewcommand{\fnum@table}[1]{\tablename~\thetable.}
\makeatother

\newtheorem{theorem}{Theorem}[section]

\newtheorem{corollary}[theorem]{Corollary}
\newtheorem{proposition}[theorem]{Proposition}
\theoremstyle{definition}
\newtheorem{definition}[theorem]{Definition}

\theoremstyle{remark}
\newtheorem{remark}[theorem]{Remark}

\newcommand{\qedflrght}{\tag*{$\square$}}

\numberwithin{equation}{section}

\title{Hom-prealternative superalgebras}

\author[1]{Ibrahima Bakayoko}
\author[2]{Sergei Silvestrov}
\affil[1]{\Affilfont D\'epartement de Math\'ematiques, Universit\'e de N'Z\'er\'ekor\'e,
\authorcr \Affilfont
BP 50 N'Z\'er\'ekor\'e, Guin\'ee.
\authorcr \Affilfont
\email{ibrahimabakayoko27@gmail.com}}
\affil[2]{\Affilfont Division of Mathematics and Physics,
School of Education, Culture and Communication,
M\"{a}lardalen University, \authorcr
\Affilfont Box 883, 72123 V{\"a}ster{\aa}s, Sweden.
\authorcr \Affilfont
\email{sergei.silvestrov@mdh.se}}

\subjclass[2020]{17D30, 17B61, 17B60, 17B62}
\keywords{Hom-prealternative superalgebra, Hom-alternative algebra, bimodule}

\date{\today}

\begin{document}

\maketitle

\abstract{
The purpose of this paper is to introduce Hom-prealternative superalgebras and their bimodules.
Some constructions of Hom-prealternative superalgebras and Hom-alternative superalgebras are given,
and their connection with Hom-alternative superalgebras are studied.
Bimodules over Hom-prealterna\-tive superalgebras are introduced, relations between bimodules over Hom-prealterna\-tive superalgebras and the bimodules of the corresponding Hom-alternative superalgebras are considered, and construction of bimodules over Hom-prealterna\-tive superalgebras by twisting is described.
}

\footnote[0]{{\it Corresponding author}: Ibrahima Bakayoko, Sergei Silvestrov}

\section{Introduction}
Hom-Lie algebras and more general quasi-Hom-Lie algebras were introduced first by Hartwig, Larsson and Silvestrov in  \cite{HartwigLarSil:defLiesigmaderiv} where a general approach to discretization of Lie algebras of vector fields using general twisted derivations ($\sigma$-deriva\-tions) and a general method for construction of deformations of Witt and Virasoro type algebras based on twisted derivations have been developed. The general quasi-Lie algebras, containing the quasi-Hom-Lie algebras and Hom-Lie algebras as subclasses, as well their graded color generalization, the color quasi-Lie algebras including color quasi-hom-Lie algebras, color hom-Lie algebras and their special subclasses the quasi-Hom-Lie superalgebras and hom-Lie superalgebras, have been first introduced in \cite{HartwigLarSil:defLiesigmaderiv,LarssonSilvJA2005:QuasiHomLieCentExt2cocyid,LarssonSilv:quasiLiealg,LSGradedquasiLiealg,LarssonSilv:quasidefsl2,SigSilv:CzechJP2006:GradedquasiLiealgWitt}.
Subsequently, various classes of Hom-Lie admissible algebras have been considered in \cite{ms:homstructure}. In particular, in \cite{ms:homstructure}, the Hom-associative algebras have been introduced and shown to be Hom-Lie admissible, that is leading to Hom-Lie algebras using commutator map as new product, and in this sense constituting a natural generalization of associative algebras as Lie admissible algebras leading to Lie algebras using commutator map. Furthermore, in \cite{ms:homstructure}, more general $G$-Hom-associative algebras including Hom-associative algebras, Hom-Vinberg algebras (Hom-left symmetric algebras), Hom-pre-Lie algebras (Hom-right symmetric algebras), and some other Hom-algebra structures, generalizing $G$-associative algebras, Vinberg and pre-Lie algebras respectively, have been introduced and shown to be Hom-Lie admissible, meaning that for these classes of Hom-algebras, the operation of taking commutator leads to Hom-Lie algebras as well. Also, flexible Hom-algebras have been introduced, connections to Hom-algebra generalizations of derivations and of adjoint maps have been noticed, and some low-dimensional Hom-Lie algebras have been described.
Since the pioneering works \cite{HartwigLarSil:defLiesigmaderiv,LarssonSilvJA2005:QuasiHomLieCentExt2cocyid,LarssonSilv:quasiLiealg,LSGradedquasiLiealg,LarssonSilv:quasidefsl2,ms:homstructure}, Hom-algebra structures have developed in a popular broad area with increasing number of publications in various directions. Hom-algebra structures are very useful since Hom-algebra structures of a given type include their classical counterparts and open broad possibilities for deformations, Hom-algebra extensions of homology and cohomology structures and representations, formal deformations of Hom-associative and Hom-Lie algebras, Hom-Lie admissible Hom-coalgebras, Hom-coalgebras, Hom-bialgebras and Hom-Hopf algebras,  \cite{
AmmarEjbehiMakhlouf:homdeformation,
BenMakh:Hombiliform,
ElchingerLundMakhSilv:BracktausigmaderivWittVir,
LarssonSilvJA2005:QuasiHomLieCentExt2cocyid,
LarssonSilvestrovGLTMPBSpr2009:GenNComplTwistDer,
MakSil:HomLieAdmissibleHomCoalgHomHopf,MakhSilv:HomAlgHomCoalg,
MakhSilv:HomDeform,
Sheng:homrep,
Yau:HomolHom,
Yau:EnvLieAlg}.
Hom-Lie algebras, Hom-Lie superalgebras, color Hom-Lie algebras, Hom-associative color algebras, Enveloping algebras of color Hom-Lie algebras,
color Hom-Leibniz algebras, omni-Hom-Lie algebras, color omni-Hom-Lie algebras, biHom-Lie algebras, biHomassociative algebras,  biHom-Frobenius algebras, Hom-Ore extensions Hom-algebras, Hom-alternative algebras, Hom-center-symmetric algebras, Hom-left-symmetric color dialgebras, Hom-dendriform algebras, Rota–Baxter Hom-algebras, Hom-tridendriform color algebras, Hom-Malcev algebras, Hom-Jordan algebras, Hom-Poisson algebras,  Color Hom-Poisson algebras, Hom-Akivis algebras, Hom-Lie-Yamaguti algebras, nearly Hom-associative algebras, Hom-Gerstenhaber algebras and Hom-Lie algebroids, $n$-Lie algebras and Hom-Nambu-Lie algebras and other $n$-ary Hom-algebra structures have been further investigated in various aspects for example in
\cite{AbAmMakh:HomaltHomMalcHomJordSuperalg,
AbAmMakh:RotaBaxterOpPreLiesuperalg,
AbdaouiAmmarMakhloufCohhomLiecolalg2015,
AbramovSilvestrov:3homLiealgsigmaderivINvol,
AmAyMabMakh:QuadrColHomLieAlgs,
AmmarEjbehiMakhlouf:homdeformation,
AmmarMakhloufHomLieSupAlg2010,
AmmarMabroukMakhloufCohomnaryHNLalg2011,
AmmarMakhloufSaadaoui2013:CohlgHomLiesupqdefWittSup,
AmmarMakhloufSilv:TernaryqVirasoroHomNambuLie,
AragonPerinaCalderonMartin2012:gradedmatrixHomalg,
AragonPerinaCalderonMartin2012:SplitRegularHomLieAlgebras,
ArmakanFarhangdoost:IJGMMP,
ArmakanSilv:envelalgcertaintypescolorHomLie,
ArmakanSilvFarh:envelopalgcolhomLiealg,
ArmakanSilvFarh:exthomLiecoloralg,
ArmakanSilv:colorHomLieHomLiebnOmniHomLie,
AtMaSi:GenNambuAlg,
akms:ternary,
ams:ternary,
ArnlindMakhloufSilvnaryHomLieNambuJMP2011,
BI:LaplacehomLiequasibialg, BI:LmodcomodhomLiequasibialg,
Bakayoko2014:ModulescolorHomPoisson,
BakayokoBangoura2015LeftHomsymHomPoisalg,
BakayokoDialo2015:genHomalgebrastr,
Bakayoko2016:HomPostLiemodulesOop,
BakyokoSilvestrov:Homleftsymmetriccolordialgebras,
BakyokoSilvestrov:MultiplicnHomLiecoloralg,
BakayokoToure2019:genHomalgebrastr,
BeitesKaygorodovPopov2018:GenDermultnaryHomOmegacolalg,
BenMakh:Hombiliform,
BenAbdeljElhamdKaygorMakhl201920GenDernBiHomLiealg,
Back2018:HomassociativeOreextensions,Back2018:HomassOreextsweakunit,Back2018:HilbertbasisthmnonassHomassOreexts,
Back2018:HilbertbasisthmnonassHomassOreexts,Back2019:formaldefsquantplanesunivenvalgs,
Back2020:MultiparformaldefsternhomNambuLiealgs,Back2020:homassociativeWeylalgebras,
CalderonMartin2020:regHomalgadmmultbasis,
CaoChen2012:SplitregularhomLiecoloralg,
ChengSu2011:CohomUnivcentrextHomLeibalg,
DassoundoSilvestrov2021:NearlyHomass,
FregierGohr2010:OnHomTypealg,
FregierGohrSilvestrov2009:UnitalalgHomasssurjinjtwist,
GaparayiIssa2012:TwistedgenerLieYamagutialg,
GrMakMenPan:Bihom1,
GuoZhZheUEPBWHLieA,
HassanzadehShapiroSutlu:CyclichomolHomasal,
ElhamdadiMakhlouf:DeformHomAltHomMalcev,
GuanChenSun:HomLieSuperalgebras,
He:stronghomassociativity,
HeMaSiUnAlHomAss,
HounkonnouDasundo:homcentersymalgbialgPropConseq2016,
HounkonnouDasundo:homcentersymalgbialg2018,
HounkonnouHoundedjiSilvestrov:DoubleconstrbiHomFrobalg,
Issa:HomAkivisalgebras,
JinLi2008:HomLiealgstrsemisimpleLiealg,
KitouniMakhloufSilvestrov:nhominduced,
kms:solvnilpnhomlie2020,
kms:narygenBiHomLieBiHomassalgebras2020,
LarssonSigSilvJGLTA2008:QuasiLiedefFttN,
LarssonSigSilvJGLTA2008,
LarssonSilvJA2005:QuasiHomLieCentExt2cocyid,
MaZheng:RotaBaxtMonoidalHomAlg,
MabroukNcibSilvestrov2020:GenDerRotaBaxterOpsnaryHomNambuSuperalgs,
Makhlouf2010:ParadigmnonassHomalgHomsuper,
MakhloufHomdemdoformRotaBaxterHomalg2011,
ms:homstructure,MakhSilv:HomDeform,
MakSil:HomLieAdmissibleHomCoalgHomHopf,MakhSilv:HomAlgHomCoalg,
MakSil:HomLieAdmissibleHomCoalgHomHopf,
MakYau:RotaBaxterHomLieadmis,
MandalMishra:HomGerstenhaberHomLiealgebroids,
MishraSilvestrov:SpringerAAS2020HomGerstenhalgsHomLiealgds,
OngongaRichterSilvestrov2019:clas3dimHomLiealg,
OngongaRichterSilvestrov2019:HomLiestr3dimskewsymalg,
OngongaRichterSilvestrov2020:clas4dimHomLiealg,
OngongaRichterSilvestrov2020:claslowdimHomLiealg,
RemmGoze2017:algvarietyHomLiealg,
RichardSilvestrovJA2008,
RichardSilvestrovGLTbdSpringer2009,
SaadaouSilvestrov:lmgderivationsBiHomLiealgebras,
ShengBai2014:homLiebialg,ShengChen2013:HomLie2algebras,
Sheng:homrep,ShengXiong:LMLA2015:OnHomLiealg,
SigSilv:GLTbdSpringer2009,
SigSilv:CzechJP2006:GradedquasiLiealgWitt,
SilvestrovParadigmQLieQhomLie2007,
SilvestrovZardeh2021:HNNextinvolmultHomLiealg,
WangZhangWei2015:HomLeibnizsuperalg,
Yau2009:HomYangBaxterHomLiequasitring,
Yau:EnvLieAlg,
Yau:HomolHom,
Yau:NoncomHomPoisalg,
Yau:HombialgHomcomodHomalg,
Yuan2012:HomLiecoloralgstr,
YauGenCom,
YauHomNambuLie,
Yau:HomaltHomJord,
ZhouChenMa:GenDerHomLiesuper}.
In particular, Color Hom-Poisson algebras \cite{Bakayoko2014:ModulescolorHomPoisson} and modules over some color Hom-algebras \cite{BakayokoDialo2015:genHomalgebrastr}, under the name of generalized Hom-algebras, have been considered. When the grading abelian group is $\mathbb{Z}_2$, the corresponding $\mathbb{Z}_2$-graded Hom-algebras are called
 Hom-superalgebras. Hom-Lie superalgebra structures such as Hom-Lie superalgebras and Hom-Lie admissible superalgebras \cite{ElhamdadiMakhlouf:DeformHomAltHomMalcev},
Rota-Baxter operator on pre-Lie superalgebras \cite{AbAmMakh:RotaBaxterOpPreLiesuperalg}, Hom-Novikov superalgebras \cite{SunB:HomNovikovsuperalg} have been considered in more details.
Hom-alternative superalgebras have been considered in \cite{AbAmMakh:HomaltHomMalcHomJordSuperalg} as a $\mathbb{Z}_2$-graded version of Hom-alternative algebras \cite{Makhl:HomaltHomJord} and their
relationships with Hom-Malcev superalgebras and Hom-Jordan superalgebras are established \cite{AbAmMakh:HomaltHomMalcHomJordSuperalg}.

The aim of this paper is to study the $\mathbb{Z}_2$-graded version of Hom-prealter\-native algebras and their bimodules. In Section~\ref{sec2:homprealtalgbimod}, we recall some basic notions on Hom-alternative superalgebras and their bimodules. We prove that bimodules over Hom-alternative
superalgebras are closed under twisting and direct product.
We show that the tensor product of super-commutative Hom-associative superalgebras and Hom-alternative superalgebras is also a Hom-alternative superalgebra.
Then we recall the definition of Hom-Jordan superalgebra.
Section~\ref{sec3:homprealtsuperalg} is devoted to Hom-prealternative superalgebras and Hom-alternative superalgebras and their connections. We point out that to any Hom-prealternative
superalgebra one may associate a Hom-alternative superalgebra, and conversly to any Hom-alternative superalgebra it corresponds a Hom-prealternative superalgebra via an $\mathcal{O}$-operator. Construction of Hom-prealternative superalgebras by composition is given.
Bimodules over Hom-prealternative superalgebras are introduced, relations between bimodules over Hom-prealternative superalgebras and bimodules of the corresponding Hom-alternative superalgebras are considered, and a construction of bimodules over Hom-prealternative superalgebras by twisting is described.

\section{Hom-prealternative algebras and bimodules} \label{sec2:homprealtalgbimod}
In this section, we present important basic notions and provide some construction results for Hom-alternative superalgebras.

Firstly, let us recall necessary important basic notions and notations on graded spaces and algebras. Throughout this paper, all linear spaces are assumed to be over a field $\mathbb{K}$ of characteristic different from $2$.

\begin{definition} Let $G$ be an abelian group. A linear space $V$ is called $G$-graded if $V=\bigoplus\limits_{a\in G} V_a$ for some family $(V_a)_{a\in G}$ of linear subspaces of $V$.
 \begin{enumerate}[label=\upshape{(\roman*)},left=7pt]
\item
An element $x\in V$ is said to be homogeneous of degree $a\in G$ if $x\in V_a$, and  $\mathcal{H}(V)= \bigcup\limits_{a\in G} V_a$ denotes the set of all homogeneous elements
in $V$.
\item Let $V=\bigoplus\limits_{a\in G} V_a$ and $V'=\bigoplus\limits_{a\in G} V'_a$ be two $G$-graded linear spaces. A linear mapping $f : V\rightarrow V'$ is said
to be homogeneous of degree $b$ if
$$f(V_a)\subseteq  V'_{a+b}, \quad \text{ for all }\quad a\in G.$$
If, $f$ is homogeneous of degree zero i.e. $f(V_a)\subseteq V'_{a}$ holds for any $a\in G$, then $f$ is said to be even.
\item 
An algebra $(A, \cdot)$ is said to be $G$-graded if its underlying linear space is $G$-graded i.e. $A=\bigoplus\limits_{a\in G}A_a$, and if furthermore
$$A_a\cdot A_b\subseteq A_{a+b}, \quad \text{ for all } \quad a, b\in G.$$
\item 
A morphism $f : A\rightarrow A'$
of $G$-graded algebras $A$ and $A'$
is by definition an algebra morphism from $A$ to $A'$, which is moreover an even mapping.
  \end{enumerate}
\end{definition}

Let $A$ be a $\mathbb{Z}_2$-graded linear space with direct sum $A=A_0\oplus A_1$. The elements of $A_j$, are said to be
homogeneous of degree (parity) $j\in\mathbb{Z}_2$. The set of all homogeneous elements of $A$ is $\mathcal{H}(A)=A_0\cup A_1$. Usually $|x|$ denotes parity of a homogeneous element $x\in \mathcal{H}(A)$.

\begin{definition}
Hom-superalgebras are triples $(A, \mu, \alpha)$ in which $A=A_0\oplus A_1$ is a $\mathbb{Z}_2$-graded linear space ($\mathbb{K}$-superspace),
 $\mu : A\times A\rightarrow A$ is an even bilinear map, and $\alpha : A\rightarrow A$ is an even linear map.
\begin{enumerate}[label=\upshape{(\roman*)},left=7pt]
\item Let $(A, \mu, \alpha)$ be a Hom-superalgebra. Hom-associator of $A$ is the even trilinear map $as_{\alpha,\mu}: A\times A\times A\rightarrow A$ given by  $$as_{\alpha,\mu}=\mu\circ(\mu\otimes\alpha-\alpha\otimes\mu).$$
In terms of elements, the map $as_{\alpha,\mu}$ is given by
$$as_{\alpha,\mu}(x, y, z)=\mu(\mu(x, y), \alpha(z))-\mu(\alpha(x), \mu(y, z)),$$
or in usual juxtaposition notation $xy=\mu(x, y)$,
$$as_{\alpha,\mu}(x, y, z)=(xy)\alpha(z)-\alpha(x)(yz).$$
\item An even linear map $f : (A, \mu, \alpha) \rightarrow (A', \mu', \alpha')$ is said to be a weak morphism of Hom-superalgebras if
$$f\circ\mu=\mu\circ(f\otimes f),$$
and a morphism of Hom-superalgebras if moreover $f\circ\alpha=\alpha'\circ f$.
\item Hom-superalgebra $(A, \mu, \alpha)$ in which $\alpha : A\rightarrow A$ is moreover an endomorphism of the algebra structure $\mu$ is said to be multiplicative, and the algebra endomorphism condition
\begin{equation}
\alpha\circ\mu=\mu\circ(\alpha\otimes\alpha)
\label{BSprealt:multiplicativityalpha}
\end{equation}
is called the multiplicativity of $\alpha$ with respect to $\mu$.
\end{enumerate}
\end{definition}

Since the grading degree of Hom-associator $|as_{\alpha,\mu}(x,y,z))| =  |x|+|y|+|z|$ for $x, y, z \in \mathcal{H}(A)=A_0\cup A_1$ in any Hom-superalgebra $(A=A_0\oplus A_1,\mu, \alpha),$
\begin{align}
as_{\alpha,\mu}(A_0, A_0, A_0)\subseteq A_0, \\
as_{\alpha,\mu}(A_1, A_0, A_0)\subseteq A_1, \\
as_{\alpha,\mu}(A_0, A_1, A_0)\subseteq A_1, \\
as_{\alpha,\mu}(A_0, A_0, A_1)\subseteq A_1, \\
as_{\alpha,\mu}(A_1, A_1, A_0)\subseteq A_0, \\
as_{\alpha,\mu}(A_1, A_0, A_1)\subseteq A_0, \\
as_{\alpha,\mu}(A_0, A_1, A_1)\subseteq A_0, \\
as_{\alpha,\mu}(A_1, A_1, A_1)\subseteq A_1.
\end{align}

\begin{definition}
Hom-associative superalgebras are those Hom-superalgebras $(A=A_0\oplus A_1, \bullet, \alpha)$ obeying super ($\mathbb{Z}_2$-graded) Hom-associativity super identity,
\begin{align}
&\forall \ x, y, z\in\mathcal{H}(A)=A_0\cup A_1: & \nonumber \\
& \quad \quad as_{\alpha,\bullet}(x, y, z)=0, &  \mbox{(super Hom-associativity)} \label{superHomassociativity}
\end{align}
equivalent in juxtaposition notation $x\bullet y=\bullet(x,y)$ to
$$(x\bullet y)\bullet\alpha(z)= \alpha(x)\bullet(y\bullet z).$$
\end{definition}
Hom-associativity super identity for Hom-superalgebras is equivalent to
\begin{equation}
as_{\alpha,\bullet}(A_i, A_j, A_k)=\{0_A\}, \quad i, j, k \in \mathbb{Z}_2.
\end{equation}

\begin{definition}
Left Hom-alternative superalgebras are Hom-superalgebras $(A=A_0\oplus A_1, \bullet, \alpha)$ obeying the left Hom-alternative super identity,
\begin{align}
&\forall \ x, y, z\in\mathcal{H}(A)=A_0\cup A_1:  \nonumber \\
& \quad \quad as_{\alpha,\bullet}(x, y, z)+(-1)^{|x||y|}as_{\alpha,\bullet}(y, x, z)=0,  \label{lefthomalternativesuperidentity}
\end{align}
equivalent in juxtaposition notation $x\bullet y=\bullet(x,y)$ to
$$(x\bullet y)\bullet\alpha(z)-\alpha(x)\bullet(y\bullet z) =
-(-1)^{|x||y|}( (y \bullet x)\bullet\alpha(z)-\alpha(y)\bullet(x\bullet z)).$$
For $(x, y, z)\in A_{|x|}\times A_{|y|} \times A_{|z|}$, $ |x|, |y|, |z| \in \mathbb{Z}_2$,
the left super Hom-alternativity for $|x||y|=0$ or $|x||y|=1$ respectively is
\begin{align}
& |x||y|=0: (x,y,z) \in  \left((A_0\times A_0) \cup (A_1 \times A_0) \cup (A_0 \times A_1)\right) \times A_{k},
\ \ k \in \mathbb{Z}_2 : \nonumber \\
&
(x\bullet y)\bullet\alpha(z)-\alpha(x)\bullet(y\bullet z) =
-((y \bullet x)\bullet\alpha(z)-\alpha(y)\bullet(x\bullet z)),
\\
& |x||y|=1: \quad (x,y,z) \in  A_1 \times A_1 \times A_k,
\ \ k \in \mathbb{Z}_2 :\nonumber \\
&
(x\bullet y)\bullet\alpha(z)-\alpha(x)\bullet(y\bullet z) =
(y \bullet x)\bullet\alpha(z)-\alpha(y)\bullet(x\bullet z).
\end{align}
\end{definition}

\begin{definition}
Right Hom-alternative superalgebra is a Hom-superalgebra $(A=A_0\oplus A_1, \bullet, \alpha)$ obeying the right Hom-alternative super identity
\begin{align}
&\forall \ x, y, z\in\mathcal{H}(A)=A_0\cup A_1:  \nonumber \\
& \quad \quad as_{\alpha,\bullet}(x, y, z)+(-1)^{|y||z|}as_{\alpha,\bullet}(x, z, y)=0, \label{righthomalternativesuperidentity}
\end{align}
which, in juxtaposition notation $x\bullet y=\bullet(x,y)$, is
$$(x \bullet y)\bullet\alpha(z)-\alpha(x)\bullet(y\bullet z) =
-(-1)^{|y||z|} ((x\bullet z)\bullet \alpha(y)-\alpha(x)\bullet (z \bullet y)).$$
For $(x, y, z)\in A_{|x|}\times A_{|y|} \times A_{|z|}$, $ |x|, |y|, |z| \in \mathbb{Z}_2$,
the left super Hom-alternativity for $|y||z|=0$ or $|y||z|=1$ respectively is
\begin{align}
& |y||z|=0: (x,y,z) \in  A_{k} \times \left((A_0\times A_0) \cup (A_1 \times A_0) \cup (A_0 \times A_1)\right),
\ \ k \in \mathbb{Z}_2, \nonumber \\
&
\quad
(x \bullet y)\bullet\alpha(z)-\alpha(x)\bullet(y\bullet z) =
-((x\bullet z)\bullet \alpha(y)-\alpha(x)\bullet (z \bullet y)),
\\
& |y||z|=1: (x,y,z) \in  A_k \times A_1 \times A_1,
\ \ k \in \mathbb{Z}_2,\nonumber \\
&
\quad
(x \bullet y)\bullet\alpha(z)-\alpha(x)\bullet(y\bullet z) =
((x\bullet z)\bullet \alpha(y)-\alpha(x)\bullet (z \bullet y)).
\end{align}
\end{definition}

\begin{definition}
Hom-alternative superalgebras are defined as both left and right Hom-alternative superalgebras.
\end{definition}

\begin{definition}
Hom-flexible superalgebra is a Hom-superalgebra $(A, \mu, \alpha)$ obeying the Hom-flexible super-identity
\begin{align}
&\forall \ x, y, z\in\mathcal{H}(A)=A_0\cup A_1:  \nonumber \\
& \quad \quad as_{\alpha,\bullet}(x, y, z)+(-1)^{|x||y|+|x||z|+|y||z|}as_{\alpha,\bullet}(z, y, x)=0, \label{homflexiblesuperidentity}
\end{align}
which, in juxtaposition notation $x\bullet y=\bullet(x,y)$, is
$$(x\bullet y)\bullet\alpha(z)-\alpha(x)\bullet(y\bullet z) =
-(-1)^{|x||y|+|x||z|+|y||z|} ((z\bullet y)\bullet\alpha(x)-\alpha(z)\bullet(y\bullet x)).$$
For $(x, y, z)\in A_{|x|}\times A_{|y|} \times A_{|z|}$, $ |x|, |y|, |z| \in \mathbb{Z}_2$,
the left super Hom-alternativity for $|x||y|+|x||z|+|y||z|=0\ \mbox{or} \ 1$ respectively is
\begin{align}
& |x||y|+|x||z|+|y||z|=0:
(x,y,z) \in
\begin{array}[t]{l}
(A_{1} \times A_0\times A_0 ) \cup (A_{0} \times A_1 \times A_0) \\
\quad \cup (A_{0} \times A_0 \times A_1) \cup (A_{0} \times A_0\times A_0),
\end{array}
\nonumber \\
&
(x\bullet y)\bullet\alpha(z)-\alpha(x)\bullet(y\bullet z) = - ((z\bullet y)\bullet\alpha(x)-\alpha(z)\bullet(y\bullet x)),
\\
& |x||y|+|x||z|+|y||z|=1:
(x,y,z) \in
\begin{array}[t]{l}
(A_{1} \times A_1\times A_0 ) \cup (A_{1} \times A_0 \times A_1) \\
\quad \cup (A_{0} \times A_1 \times A_1) \cup (A_{1} \times A_1\times A_1),
\end{array}
\nonumber \\
&
(x\bullet y)\bullet\alpha(z)-\alpha(x)\bullet(y\bullet z) =
(z\bullet y)\bullet\alpha(x)-\alpha(z)\bullet(y\bullet x).
\end{align}
\end{definition}

\begin{definition}
A bimodule over a Hom-alternative superalgebra $(A, \bullet, \alpha)$ consists of a $\mathbb{Z}_2$-graded linear space $V$ with an even linear map $\beta : V\rightarrow V$ and two even bilinear maps
\begin{eqnarray*}
 \succ\; :\quad  A\otimes V&\rightarrow& V \hspace{1,3cm} \prec \;:\quad V\otimes A\rightarrow V \\
x\otimes v&\mapsto&x\succ v \hspace{2cm}  v\otimes x\mapsto v\prec x
 \end{eqnarray*}
 such that, for any homogeneous elements $x, y\in A$ and  $v\in V$,
\begin{align}
&\begin{array}{l}
(v\prec x)\prec\alpha(y) +(-1)^{|x||v|}(x\succ v)\prec\alpha(y)-\\
\qquad (-1)^{|x||v|}\alpha(x)\succ(v\prec y) -\beta(v)\prec(x\bullet y)=0,
\end{array} \label{abm1}\\
&\begin{array}{l}
\alpha(y)\succ(v\prec x) -(y\succ v)\prec\alpha(x)- \\
\qquad (-1)^{|x||v|}(y\bullet x)\succ\beta(v)+ (-1)^{|x||v|}\alpha(y)\succ(x\succ v)=0,
\end{array} \label{abm2}\\
&\begin{array}{l}
(x\bullet y)\succ\beta(v) +(-1)^{|x||y|}(y\bullet x)\succ\beta(v)-\\
\qquad \alpha(x)\succ(y\succ v) -(-1)^{|x||y|}\alpha(y)\succ(x\succ v)=0,
\end{array} \label{abm3}\\
&\begin{array}{l}
\beta(v)\prec(x\bullet y) +(-1)^{|x||y|}\beta(v)\prec(y\bullet x)-\\
\qquad (v\prec x)\prec\alpha(y)- (-1)^{|x||y|}(v\prec y)\prec\alpha(x)=0.
\end{array}
\label{abm4}
\end{align}
\end{definition}
\begin{remark}
 The notation $x\succ v$ means the left action of $x$ on $v$ and
$v\prec x$ means the right action of $x$ on $v$ given by the
linear operators on $V$ defined by
$$L_\succ(x)v=x\succ v, \quad \quad R_\prec(x)v=v\prec x.$$
\end{remark}

Bimodules over Hom-alternative superalgebras are closed under twisting in the sense of Theorem \ref{thm:bimodulehomalttwisting}.
\begin{theorem} \label{thm:bimodulehomalttwisting}
 Let $(V, L_\succ, R_\prec, \beta)$ be a bimodule over the multiplicative Hom-alternative superalgebra
$(A, \bullet, \alpha)$. Then,
$(V, L_\succ^\alpha, R_\prec^\alpha, \beta)$ is a bimodule over $A$, where
$L_\succ^\alpha=L_\succ\circ(\alpha^2\otimes Id)$ and $R_\prec^\alpha=R_\prec\circ(\alpha^2\otimes Id).$
\end{theorem}
\begin{proof}
We only prove \eqref{abm1}, as \eqref{abm2}, \eqref{abm3}, \eqref{abm4} are proved similarly.
With
 \begin{align*}
  x\succeq v&=L_\succ^\alpha(x)v=L_\succ\circ(\alpha^2\otimes Id)(x\otimes v)=\alpha^2(x)\succ v,\\
v\preceq x&=R_\prec^\alpha(x)v=R_\prec\circ(id\otimes\alpha^2)(v\otimes x)=v\prec\alpha^2(x),
 \end{align*}
for any $x, y\in A$ and any $v\in V$,
\begin{align*}
&(v\preceq x)\preceq\alpha(y)+(-1)^{|x||v|}(x\succeq v)\preceq\alpha(y)\\
& \qquad -(-1)^{|x||v|}\alpha(x)\succeq(v\preceq y) -\beta(v)\preceq(x\bullet y)= \\
&(v\prec\alpha^2(x))\prec\alpha^3(y)+(-1)^{|x||v|}(\alpha^2(x)\succ v)\prec\alpha^3(y) \\
& \qquad -(-1)^{|x||v|}\alpha^3(x)\succ(v\prec\alpha^2(y))-\beta(v)\prec\alpha^2(x\bullet y) = \\
&(v\prec\alpha^2(x))\prec\alpha^3(y)+(-1)^{|x||v|}(\alpha^2(x)\succ v)\prec\alpha^3(y) \\
& \qquad -(-1)^{|x||v|}\alpha^3(x)\succ(v\prec\alpha^2(y))-\beta(v)\prec(\alpha^2(x)\bullet \alpha^2(y)) = 0,
\end{align*}
by using the multiplicativity of $\alpha$ in the last term, and then \eqref{abm1} for $\alpha^2(x)$ and $\alpha^2(y)$ in $(V, L_\succ, R_\prec, \beta)$.
\end{proof}

For two $\mathbb{Z}_2$-graded linear spaces $V=\oplus_{a\in \mathbb{Z}_2} V_a$ and $V'=\oplus_{a\in \mathbb{Z}_2} V'_a$, the tensor product $V\otimes V'$ is also a $\mathbb{Z}_2$-graded linear space such that for $\alpha\in \mathbb{Z}_2$,
$$(V\otimes V')_\alpha=\sum_{\alpha=a+a'}V_a\otimes V_{a'}.$$

\begin{theorem}
 Let $(A, \bullet, \alpha)$ be a super-commutative Hom-associative  superalgebra and
 $({A}', \bullet', \alpha')$ be a Hom-alternative superalgebra.
Then the tensor product
$(A\otimes{A}', \ast, \alpha\otimes\alpha')$ where for $x, y\in\mathcal{H}(A), a, b\in\mathcal{H}(A')$,
\begin{eqnarray*}
 (\alpha\otimes\alpha')(x\otimes a)&=& \alpha(x)\otimes\alpha'(a), \\
(x\otimes a)\ast(y\otimes b)&=& (-1)^{|a||y|}(x\bullet y)\otimes (a\bullet' b),
\end{eqnarray*}
is a Hom-alternative superalgebra
\end{theorem}
\begin{proof}
 Let us set $X=x\otimes a,\ Y=y\otimes b,\ Z=z\otimes c \ \in \mathcal{H}(A)\otimes\mathcal{H}(A')$.
 Then,
\begin{align*}
& as_{\alpha\otimes\alpha',\ast}(X, Y, Z)=as_{\alpha\otimes\alpha',\ast}
(x\otimes a, y\otimes b, z\otimes c) \\
& = ((x\otimes a)\ast(y\otimes b))\ast (\alpha\otimes\alpha')(z\otimes c)-
 (\alpha\otimes\alpha')(x\otimes a)\ast ((y\otimes b)\ast (z\otimes c))\\
& = \Big((x\otimes a)\ast(y\otimes b))\ast (\alpha(z)\otimes\alpha'(c))
-(\alpha(x)\otimes\alpha'(a))\ast ((y\otimes b)\ast (z\otimes c)\Big)\\
&  = (-1)^{|a||y|}\Big((x\bullet y)\otimes (a\bullet' b))\ast (\alpha(z)\otimes\alpha'(c)) \\
& \hspace{2cm} -(-1)^{|b||z|}(\alpha(x)\otimes\alpha'(a))\ast((y\bullet z)\otimes (b \bullet' c)\Big)\\
&  = (-1)^{|a||y|+(|a|+|b|)|z|}
((x\bullet y)\bullet \alpha(z))\otimes ((a\bullet' b)\bullet' \alpha'(c))\\
& \hspace{2cm} -(-1)^{|b||z|+|a|(|y|+|z|)}
(\alpha(x)\bullet (y\bullet z))\otimes (\alpha'(a)\bullet' (b\bullet' c)). \\
& as_{\alpha\otimes\alpha',\ast}(X, Y, Z)+(-1)^{|XY|}as_{\alpha\otimes\alpha',\ast}(Y, X, Z)  \\
&= as_{\alpha\otimes\alpha',\ast}(x\otimes a, y\otimes b, z\otimes c) +
(-1)^{|(x\otimes a)\ast (y\otimes b)|}as_{\alpha\otimes\alpha',\ast}(x\otimes a, y\otimes b, z\otimes c)\\
& = (-1)^{|a||y|+(|a|+|b|)|z|}((x\bullet y)\bullet \alpha(z))\otimes((a\bullet' b)\bullet' \alpha'(c)) \\
& \quad \quad -(-1)^{|b||z|+|a|(|y|+|z|)}(\alpha(x)\bullet (y\bullet z))\otimes (\alpha'(a)\bullet' (b\bullet' c)) \\
& \quad \quad +(-1)^{(|x|+|a|)(|y|+|b|)+|b||x|+(|a|+|b|)|z|}((y\bullet x)\bullet \alpha(z))\otimes ((b\bullet' a)\bullet' \alpha'(c)) \\
& \quad \quad -(-1)^{(|x|+|a|)(|y|+|b|)+|a||z|+|b|(|x|+|z|)}(\alpha(y)\bullet (x\bullet z))\otimes
(\alpha'(b)\bullet' (a\bullet' c)).
\end{align*}
By super-commutativity, $x\bullet y=(-1)^{|x||y|}y\bullet x$, and Hom-associativity  \eqref{superHomassociativity},
\begin{align*}
& as_{\alpha\otimes\alpha',\ast}(X, Y, Z)+(-1)^{|XY|}as_{\alpha\otimes\alpha',\ast}(Y, X, Z)= \\
& (-1)^{|a||y|+|a||z|+|b||z|}\Big((x\bullet y)\bullet \alpha(z)-\alpha(x)\bullet (y\bullet z) \\
& \hspace{4cm} +(-1)^{|x||y|}(y\bullet x)\bullet \alpha(z)-(-1)^{|x||y|}\alpha(y)\bullet (x\bullet z)\Big) \\
& \hspace{8cm} \otimes(a\bullet' b)\bullet' \alpha'(c).
\end{align*}
The left hand side vanishes by the left Hom-alternativity of $A'$. The right Hom-alternativity is proved similarly.
\end{proof}

\begin{definition}[\cite{BakayokoToure:somecolorHomalgstr}] \label{bk3}
{\it Left averaging operator} over a Hom-alternative superalgebra  $(A, \cdot, \alpha)$ is an even linear map $\beta : A\rightarrow A$ satisfying
\begin{eqnarray*}
\alpha\circ\beta &=& \beta\circ\alpha, \\
\beta(x)\cdot \beta(y) &=& \beta(\beta(x)\cdot y) \quad  \mbox{for all}\; x, y\in\mathcal{H}(A).
\end{eqnarray*}
{\it Right averaging operator} over a Hom-alternative superalgebra  $(A, \cdot, \alpha)$ is an even
linear map $\beta: A\rightarrow A$ such that $\alpha\circ\beta=\beta\circ\alpha$ and
$$\beta(x)\cdot \beta(y)=\beta(x\cdot \beta(y)) \;\;\;
\mbox{for all} \; x, y\in\mathcal{H}(A).$$
{\it Averaging operator} over a Hom-alternative superalgebra  $(A, \cdot, \alpha)$ is both left averaging operator and right averaging operator, meaning an even
linear map $\beta: A\rightarrow A$ such that $\alpha\circ\beta=\beta\circ\alpha$ and
$$\beta(\beta(x)\cdot y)= \beta(x)\cdot \beta(y)=\beta(x\cdot \beta(y)).$$
\end{definition}

\begin{proposition}
Let $(A, \cdot, \alpha)$ be a Hom-alternative  algebra. Let $\beta : A\rightarrow A$ be an
element of the centroid, an even linear map such that
for all $x, y\in\mathcal{H}(A),$
\begin{align}
& \beta\circ\alpha=\alpha\circ\beta, \label{centroid:alphabetacomut}\\
& \beta(x\cdot y)=\beta(x)\cdot y = x \cdot \beta(y). \label{centroid}
\end{align}
Then
 $(A, \cdot_\beta = \beta\circ\cdot, \alpha)$
is a Hom-alternative superalgebra.
\end{proposition}
\begin{proof} Hom-alternativity means both left and right Hom-alternativity.
The left and the right Hom-alternativity
of $(A, \cdot_\beta = \beta\circ\cdot, \alpha)$ are proved respectively as follows.
For any $x, y, z\in\mathcal{H}(A)$,
\begin{align}
as_{\alpha,\cdot_\beta}(x,y,z)&= (x\cdot_\beta y)\cdot_\beta\alpha(z)-\alpha(x)\cdot_\beta (y\cdot_\beta z) \nonumber \\
&=\beta (\beta(x\cdot y)\cdot\alpha(z))-\beta (\alpha(x)\cdot\beta (y\cdot z)) \nonumber\\
&\stackrel{\eqref{centroid}}{=}
\beta((\beta(x)\cdot y))\cdot\alpha(z)-\beta (\alpha(x))\cdot(\beta (y)\cdot z) \nonumber\\
&\stackrel{\eqref{centroid}}{=}
(\beta(x)\cdot \beta(y))\cdot \alpha(z)-
\beta (\alpha(x))\cdot(\beta(y) \cdot z)  \nonumber\\
&\stackrel{\eqref{centroid:alphabetacomut}}{=} (\beta(x)\cdot\beta(y))\cdot\alpha(z)
-\alpha (\beta(x))\cdot(\beta(y) \cdot z)   \nonumber\\
&=
as_{\alpha,\cdot}(\beta(x),\beta(y),z) \label{altcompmultproofstep} \\
&\text{(using left Hom-alternativity of $(A, \cdot, \alpha)$)} \nonumber \\
&\stackrel{\eqref{lefthomalternativesuperidentity}}{=}
-(-1)^{|x||y|}  as_{\alpha,\cdot}(\beta(y),\beta(x),z)  \nonumber\\
&\text{(using proved in \eqref{altcompmultproofstep} for $(y,x,z)$)} \nonumber \\
&=
-(-1)^{|x||y|}as_{\alpha,\cdot_\beta}(y,x,z),   \nonumber
\\
as_{\alpha,\cdot_\beta}(x,y,z) &=
(x\cdot_\beta y)\cdot_\beta\alpha(z)-\alpha(x)\cdot_\beta (y\cdot_\beta z) \nonumber \\
&=\beta (\beta(x\cdot y)\cdot\alpha(z))-\beta (\alpha(x)\cdot\beta (y\cdot z))
\nonumber \\
&\stackrel{\eqref{centroid}}{=}
(\beta(x)\cdot y)\cdot \beta(\alpha(z))-
\beta (\alpha(x))\cdot(y \cdot \beta(z))
\nonumber \\
&\stackrel{\eqref{centroid:alphabetacomut}}{=}
(\beta(x)\cdot y)\cdot\alpha(\beta(z))-\alpha (\beta(x))\cdot(y \cdot \beta(z))
\nonumber \\
&=
as_{\alpha,\cdot}(\beta(x),y,\beta(z)) \nonumber \\
&\text{(using right Hom-alternativity of $(A, \cdot, \alpha)$)} \nonumber \\
&\stackrel{\eqref{righthomalternativesuperidentity}}{=}
-(-1)^{|y||z|}  as_{\alpha,\cdot}(\beta(x),\beta(z),y) \nonumber \\
&\text{(using proved in \eqref{altcompmultproofstep} for $(x,z,y)$)} \nonumber \\
&=
-(-1)^{|y||z|}as_{\alpha,\cdot_\beta}(x,z,y).  \nonumber
\qedflrght \end{align}
\phantom{\qedhere}
\end{proof}
\begin{proposition}
 Any Hom-alternative superalgebra $(A, \cdot, \alpha)$ with an averaging operator $\partial$ is a Hom-alternative superalgebra with respect to multiplication $\ast : A\times A\rightarrow A$ defined by $x\ast y:=x\cdot\partial(y)$ and the same twisting map $\alpha$.
\end{proposition}
\begin{proof}
For any $x, y, z\in\mathcal{H}(A)$,
\begin{align*}
(x\ast y)\ast\alpha(z)-\alpha(x)\ast(y\ast z)
&=\alpha(x)\cdot (\partial(y)\cdot\partial(z))-\alpha(x)\cdot\partial(y\cdot\partial(z)) \\
&=\alpha(x)\cdot (\partial(y)\cdot\partial(z))-\alpha(x)\cdot(\partial(y)\cdot\partial(z))
=0.
\end{align*}
On the one hand, exchanging the role of $x$ and $y$, yields
$$
(x\ast y)\ast\alpha(z)-\alpha(x)\ast(y\ast z)+(-1)^{|x||y|}\Big((y\ast x)\ast\alpha(z)-\alpha(y)\ast(x\ast z)\Big)=0.
$$
On the other hand, exchanging the role of $y$ and $z$, yields
$$
(x\ast y)\ast\alpha(z)-\alpha(x)\ast(y\ast z)+(-1)^{|y||z|}\Big((x\ast z)\ast\alpha(y)-\alpha(x)\ast(z\ast y)\Big)=0.
$$
This completes the proof.
\end{proof}

\begin{definition}[\cite{AbAmMakh:HomaltHomMalcHomJordSuperalg}]
\label{hjd}
 A Hom-Jordan superalgebra is a Hom-superalgebra $(A, \bullet, \alpha)$ satisfying
 super-commutativity and Hom-Jordan super identity
 \begin{align}
& \forall \ x, y, z, t\in\mathcal{H}(A):  \nonumber \\
& x\bullet y=(-1)^{|x||y|}y\bullet x, &  \mbox{super-commutativity} \\
& \sum_{\circlearrowleft(x,y,t)}(-1)^{|t|(|x|+|z|)}as_{\bullet,\alpha}(xy, \alpha(z), \alpha(t))=0,
& \begin{array}[t]{l}
\mbox{Hom-Jordan } \\
\mbox{super identity}
\end{array}
\label{homjordansuperidentity}
\end{align}
where $\displaystyle{\sum_{\circlearrowleft(a,b,c)}}$ is the summation over cyclicly permutated $(a, b, c)$.
Hom-Jordan super identity \eqref{homjordansuperidentity} in juxtaposition notation
$x\bullet y=\bullet(x,y)$ is
\begin{align*}
& \forall \ x, y, z, t\in\mathcal{H}(A):  \\
&
\begin{array}[t]{c}
\displaystyle{\sum_{\circlearrowleft(x,y,t)}}(-1)^{|t|(|x|+|z|)} ((x\bullet y)\bullet \alpha(z))\bullet\alpha^2(t)= \\
\hspace{4cm} \displaystyle{\sum_{\circlearrowleft(x,y,t)}}(-1)^{|t|(|x|+|z|)} \alpha(x\bullet y)\bullet(\alpha(z) \bullet \alpha(t)).
\end{array}
\end{align*}
\end{definition}

\begin{remark}
If $(x,y,z,t)\in (A_0\times A_0\times A_0\times A_0) \cup (A_1\times A_1 \times A_1 \times A_1)$, then $(-1)^{|t|(|x|+|z|)}=(-1)^{|x|(|y|+|z|)}=(-1)^{|y|(|t|+|z|)}=1$, and Hom-Jordan super identity is
\begin{eqnarray}
\sum_{\circlearrowleft(x,y,t)}((x\bullet y) \bullet \alpha(z))\bullet\alpha^2(t)=
\sum_{\circlearrowleft(x,y,t)} \alpha(x\bullet y)\bullet(\alpha(z) \bullet \alpha(t)).
\end{eqnarray}
\end{remark}

\begin{theorem}[\cite{AbAmMakh:HomaltHomMalcHomJordSuperalg}]
\label{hj}
Any multiplicative Hom-alternative superalgebra is Hom-Jordan admissible, that is, for any multiplicative Hom-alternative superalgebra $(A, \bullet, \alpha)$, the Hom-superalgebra
$A^+=(A, \ast, \alpha)$ is a multiplicative Hom-Jordan superalgebra, where
$x\ast y=xy+(-1)^{|x||y|}yx$.
\end{theorem}

\section{Hom-prealternative and Hom-alternative superalgebras} \label{sec3:homprealtsuperalg}
In this section, we introduce Hom-prealternative superalgebras, give some contruction theorems and study their connection with Hom-alternative superalgebras. The associated bimodules are also discussed.
\subsection{Prealternative superalgebras}
\begin{definition} \label{pad}
A Hom-prealternative superalgebra is a quadruple $(A, \prec, \succ, \alpha)$ in which $A$ is a supervector space,
 $\prec, \succ : A\otimes A\rightarrow A$ are even bilinear maps and $\alpha : A\rightarrow A$ an even linear map such that,
for any $x, y, z\in\mathcal{H}(A)$,
\begin{align}
& \begin{array}{l} (x\bullet x)\succ\alpha(y)-\alpha(x)\succ(x\succ y)=0, \end{array}
\label{pa1}\\
& \begin{array}{l} (x\prec y)\prec\alpha(y)-\alpha(x)\prec(y\bullet y)=0, \end{array}
\label{pa2}\\
&\begin{array}{l}
(x\succ y)\prec\alpha(z) -\alpha(x)\succ(y\prec z)+ \\
\quad (-1)^{|x||y|}(y\prec x)\prec\alpha(z)-(-1)^{|x||y|}\alpha(y)\prec(x\bullet z)=0,
\end{array}
\label{pa3}\\
&
\begin{array}{l}
(x\succ y)\prec\alpha(z) -\alpha(x)\succ(y\prec z)+  \\
\quad (-1)^{|y||z|}(x\bullet z)\succ\alpha(y)-(-1)^{|y||z|}\alpha(x)\succ(z\succ y)=0,
\end{array}
 \label{pa4}
\end{align}
where $x\bullet y=x\succ y+x\prec y$.
\end{definition}
\begin{definition}
Let $(A, \prec, \succ, \alpha)$ and $(A', \prec', \succ', \alpha')$ be two Hom-prealternative superalgebras.
An even linear map $f : A\rightarrow A'$ is said to be a morphism of Hom-prealternative superalgebras if,
for any $x, y\in\mathcal{H}(A)$,
\begin{eqnarray*}
 \alpha'\circ f=f\circ\alpha,\;\; f(x\prec y)=f(x)\prec'f(y) \;\;\mbox{and}\;\; f(x\succ y)=f(x)\succ'f(y).
\end{eqnarray*}
A Hom-prealternative superalgebra $(A, \prec, \succ, \alpha)$ in which $\alpha : A\rightarrow A$ is a morphism is called a multiplicative Hom-alternative
superalgebra.
\end{definition}
\begin{remark}
Axioms \eqref{pa1} and \eqref{pa2} can be rewritten respectively as
\begin{eqnarray}
\begin{aligned}
(x\bullet y)\succ\alpha(z) &-\alpha(x)\succ(y\succ z)+ \\
& (-1)^{|x||y|}(y\bullet x)\succ\alpha(z)-(-1)^{|x||y|}\alpha(y)\succ(x\succ z)=0,
\end{aligned} \label{pa5}\\
\begin{aligned}
(x\prec y)\prec\alpha(z) &-\alpha(x)\prec(y\bullet z)+\\
& (-1)^{|y||z|}(x\prec z)\prec\alpha(y)-(-1)^{|y||z|}\alpha(x)\prec(z\bullet y)=0.
\end{aligned}
\label{pa6}
\end{eqnarray}
\end{remark}

\begin{remark}
  If $(A, \prec, \succ, \alpha)$ is a Hom-prealternative superalgebra, then so is $(A, \prec_\lambda=\lambda\cdot \prec,
\succ_\lambda=\lambda\cdot\succ, \alpha)$.
\end{remark}
Using the following notations \cite{NBPrealtalg}, \cite{Felipe:DendiformalgRotabaxt2013}:
\begin{eqnarray}
 (x, y, z)_1&=&(x\bullet y)\succ\alpha(z)-\alpha(x)\succ(y\succ z),\\
(x, y, z)_2&=&(x\succ y)\prec\alpha(z)-\alpha(x)\succ(y\prec z),\\
(x, y, z)_3&=&(x\prec y)\prec\alpha(z)-\alpha(x)\prec(y\bullet z),
\end{eqnarray}
the axioms in Definition \ref{pad} of Hom-prealternative superalgebras can be rewritten for any $x, y, z\in\mathcal{H}(A)$ as
\begin{eqnarray}
(x, x, z)_1=(y, x, x)_3&=&0\\
(x, y, z)_2+(-1)^{|x||y|}(y, x, z)_3 &=&0,\\
(x, y, z)_2+  (-1)^{|y||z|}(x, z, y)_1 &=&0.
\end{eqnarray}

The following definition is motivated by \cite[Definition 17, Definition 18]{Felipe:DendiformalgRotabaxt2013}.
\begin{definition}
 A Hom-prealternative superalgebra $(A, \prec, \succ, \alpha)$ is said to be left Hom-alternative if
 \begin{eqnarray}
  (x, y, z)_i+(-1)^{|x||y|}(y, x, z)_i=0, \quad i=1, 2, 3.
 \end{eqnarray}
and right Hom-alternative if
 \begin{eqnarray}
  (x, y, z)_i+(-1)^{|y||z|}(x, z, y)_i=0, \quad i=1, 2, 3.
 \end{eqnarray}
\end{definition}
\begin{definition}
 A Hom-prealternative superalgebra algebra $(A, \prec, \succ, \alpha)$ is said to be flexible if
\begin{eqnarray}
 (x, y, x)_i=0, \quad i=1, 2, 3.
\end{eqnarray}
\end{definition}
\begin{theorem}
If $(A, \prec, \succ, \alpha)$ is a left Hom-prealternative superalgebra, then
 $Alt(A)=(A, \bullet, \alpha)$ is a left Hom-alternative superalgebra. If $(A, \prec, \succ, \alpha)$ is a right Hom-prealternative superalgebra, then
 $Alt(A)=(A, \bullet, \alpha)$ is a right Hom-alternative superalgebra.
\end{theorem}
\begin{proof}
For any $x, y, z\in\mathcal{H}(A)$,
 \begin{align*}
& as_\bullet(z, x, y) =(z\bullet x)\bullet\alpha(y)-\alpha(z)\bullet(x\bullet y)\\
& \qquad =(z\bullet x)\prec\alpha(y)+(z\bullet x)\succ\alpha(y)-\alpha(z)\prec(x\bullet y)-\alpha(z)\succ(x\bullet y)\\
& \quad =(z\prec x+z\succ x)\prec\alpha(y)+(z\bullet x)\succ\alpha(y)- \\
& \qquad \qquad \qquad   \alpha(z)\prec(x\bullet y)-\alpha(z)\succ(x\prec y+x\succ y)\\
& \quad =((z\prec x)\prec\alpha(y)-\alpha(z)\prec(x\bullet y))+
((z\succ x)\prec\alpha(y)-\alpha(z)\succ(x\prec y)) + \\
& \qquad \qquad \qquad
((z\bullet x)\succ\alpha(y)-\alpha(z)\succ(x\succ y))\\
& \quad =(z, x, y)_3+(z, x, y)_2+(z, x, y)_1\\
& \quad =-(-1)^{|x||y|}((z, y, x)_3+(z, y, x)_2+(z, y, x)_1) \\
& \quad =-(-1)^{|x||y|}  as_\bullet(z, y, x).
 \end{align*}
The left alternatively is proved analogously.
\end{proof}
Note that the left and right Hom-alternativity for dialgebras is not defined in the same way that the one of algebras with one operation;
so the two terminologies must not be confused.
\begin{proposition}
 Let $(A, \prec, \succ, \alpha)$ be a fexible Hom-prealternative superalgebra. Then $(A, \bullet, \alpha)$ is a flexible Hom-alternative superalgebra.
\end{proposition}

\begin{theorem} \label{hk}
Let $(A, \prec, \succ, \alpha)$ be a Hom-prealternative superalgebra.
Then $A'=(A, \prec', \succ', \alpha)$ is also a
Hom-prealternative superalgebra with
\begin{eqnarray*}
x\prec' y &=& (-1)^{|x||y|}y\succ x, \\
 x\succ'y &=& (-1)^{|x||y|}y\prec x.
\end{eqnarray*}
\end{theorem}
\begin{proof}
We prove only \eqref{pa3}, as \eqref{pa1}, \eqref{pa2} and \eqref{pa4} are proved similarly.
For any $x, y, z\in\mathcal{H}(A)$,
\begin{multline*}
\qquad(x\succ' y)\prec'\alpha(z)-\alpha(x)\succ'(y\prec z)+\\
(-1)^{|x||y|}(y\prec' x)\prec'\alpha(z)-(-1)^{|x||y|}\alpha(y)\prec'(x\bullet' z) \\
=(-1)^{|x||y|}(y\prec x)\prec'\alpha(z)-(-1)^{|y||z|}\alpha(z)\succ'(z\succ y)+ \\
(x\succ y)\prec'\alpha(z)-(-1)^{|x||y|+|x||z|}\alpha(y)\prec'(z\bullet x) \\
=(-1)^{|x||y|+(|x|+|y|)|z|}\alpha(z)\succ(y\prec x)-(-1)^{|y||z|+|x|(|y|+|z|)}(z\succ y)\prec\alpha(x)+ \\
\quad(-1)^{(|x|+|y|)|z|}\alpha(z)\succ(x\succ y)-(-1)^{|x||y|+|x||z|+|y|(|x|+|z|)}(z\bullet x)\succ\alpha(y) \\
=(-1)^{|x||y|+|x||z|+|y||z|}[\alpha(z)\succ(y\prec x)-(z\succ y)\prec\alpha(x)+ \\
(-1)^{|x||y|}\alpha(z)\succ(x\succ y)- (-1)^{|x||y|}(z\bullet x)\succ\alpha(y)] =0
\end{multline*}
by axiom \eqref{pa4} for $(A, \prec, \succ, \alpha)$.
\end{proof}

Note that $Alt(A')=Alt(A)^{op}$, that is, $x\bullet'y=(-1)^{|x||y|}y\bullet x$, for $x, y\in\mathcal{H}(A)$.
\begin{theorem}\label{hl}
 Let $(A, \prec, \succ, \alpha)$ be a Hom-prealternative superalgebra. Let us define the operation
$x\bullet y=x\prec y+x\succ y$ for any homogeneous elements $x, y$ in $A$.
Then $Alt(A)=(A, \bullet, \alpha)$ is a Hom-alternative superalgebra.
\end{theorem}
\begin{proof}
 Let us prove the left alternativity. For any homogeneous $x, y, z\in A$,
\begin{multline*}
as_\bullet(x, y, z)+(-1)^{|x||y|}as_\bullet(y, x, z)= \\
(x\prec y)\prec\alpha(z)+(x\succ y)\prec\alpha(z)+ (x\bullet y)\succ\alpha(z) -\alpha(x)\prec(y\bullet z) \\
-\alpha(x)\succ(y\succ z)-\alpha(x)\succ(y\prec z)+(-1)^{|x||y|}(y\prec x)\prec\alpha(z) \\
+(-1)^{|x||y|}(y\succ x)\prec\alpha(z)+(-1)^{|x||y|}(y\bullet x)\succ\alpha(z)
-(-1)^{|x||y|}\alpha(y)\prec(x\bullet z) \\
-(-1)^{|x||y|}\alpha(y)\succ(x\prec z)-(-1)^{|x||y|}\alpha(y)\succ(x\succ )z.
\end{multline*}
The left hand side vanishes by using one axiom \eqref{pa3} and twice axiom \eqref{pa6}.
\end{proof}

The Hom-alternative superalgebra $Alt(A)=(A, \bullet, \alpha)$ in Theorem \ref{hl} is called the associated Hom-alternative superalgebra of $(A, \prec, \succ, \alpha)$.
We call $(A, \prec, \succ, \alpha)$ a compatible Hom-prealternative superalgebra
structure on the Hom-alternative superalgebra $Alt(A)$.

Theorem \ref{hj}, Theorem \ref{hk} and Theorem \ref{hl} yield the following corollary.
\begin{corollary}
 Let $(A, \prec, \succ, \alpha)$ be a multiplicative Hom-prealternative superalgebra. Then $(A, \ast, \alpha)$ is a  multiplicative
 Hom-Jordan superalgebra with
$$x\ast y=x\prec y+x\succ y+ (-1)^{|x||y|}y\prec x+(-1)^{|x||y|}y\succ x.$$
\end{corollary}

Let us define the notion of $\mathcal{O}$-operator for Hom-alternative superalgebras.
\begin{definition}
 Let $(V, L, R, \beta)$ be a bimodule of the Hom-alternative superalgebra $(A, \bullet, \alpha)$. An even linear map $T: V\rightarrow A$ is called
an $\mathcal{O}$-operator associated to $(V, L, R, \beta)$ if for any $u, v\in V$,
 \begin{eqnarray}
  T(u)\bullet T(v)&=& T(L(T(u))v+R(T(v))u), \\
  T\circ\beta &=& \alpha\circ T.
 \end{eqnarray}
\end{definition}

\begin{theorem}
 Let $T: V\rightarrow A$ be  an $\mathcal{O}$-operator of the Hom-alternative superalgebra $(A, \bullet, \alpha)$ associated to the bimodule
 $(V, L, R, \beta)$.  Then $(V, \prec, \succ, \beta)$ is a Hom-prealternative superalgebra structure, where for all $u, v\in V$,
$$u\prec v=R(T(v))u\quad\mbox{and}\quad u\succ v=L(T(u))v.$$
Therefore, $(V, \bullet=\prec+\succ, \beta)$ is the associated Hom-alternative superalgebra of this Hom-prealternative superalgebra,
and $T$ is a homomorphism of Hom-alternative superalgebras.
Furthermore, $T(V)=\{T(v), v\in V\}\subseteq A$ is a Hom-alternative subalgebra of $(A, \bullet, \alpha)$, and $(T(V), \prec, \succ, \alpha)$ is a
Hom-prealternative superalgebra, where for all $u, v\in V$,
$$T(u)\prec T(v)=T(u\prec v)\quad\mbox{and}\quad T(u)\succ T(v)=T(u\succ v).$$
The associated Hom-alternative superalgebra $(T(V), \bullet=\prec+\succ, \alpha)$ is just the Hom-alternative subalgebra structure of
$(A, \bullet, \alpha)$, and $T$ is a homomorphism of Hom-prealternative superalgebras.
\end{theorem}
\begin{proof}
 For any homogeneous elements $u, w, w\in V$,
\begin{multline*}
 (u\succ v)\prec\beta(w)-\beta(u)\succ(v\prec w)+ (-1)^{|u||v|}(v\prec u)\prec\beta(w)-\\
 (-1)^{|u||v|}\beta(v)\prec(u\bullet w)\\
=(T(u)v)T\beta(w)-T\beta(u)(vT(w))+(-1)^{|u||v|}(vT(u))T\beta(w)- \\
 (-1)^{|u||v|}\beta(v)T(uT(w)+T(u)w)\\
=(T(u)v)\alpha(T(w))-\alpha(T(u))(vT(w))+ \\
(-1)^{|u||v|}(vT(u))\alpha(T(w))-(-1)^{|u||v|}\beta(v)(T(u)T(w))
=0. \quad(\mbox{by}\; \eqref{abm1})
\end{multline*}
The other identities are checked similarly, and the rest of the proof is easy.
\end{proof}
\begin{definition}
A Hom-alternative Rota-Baxter superalgebra of weight $\lambda$ is a Hom-alternative superalgebra
 $(A, \cdot, \alpha)$ together with an even linear self-map $R : A\rightarrow A$ such that $R\circ\alpha=\alpha\circ R$ and
\begin{equation}
R(x)\cdot R(y) = R\Big(R(x)\cdot y + x \cdot R(y) +\lambda x\cdot y\Big).
\end{equation}
\end{definition}
\begin{corollary}
 Let $(A, \cdot, \alpha)$ be a Hom-alternative superalgebra and $R: A\rightarrow A$ a Rota-Baxter operator of weight $0$ on $A$. Then
\begin{enumerate}[label=\upshape{(\roman*)},left=7pt]
\item
$A_R=(A, \prec, \succ, \alpha)$ is a Hom-prealternative superalgebra, where
$$x\prec y=x\cdot R(y)\quad\mbox{and}\quad x\succ y=R(x)\cdot y,$$
for any homogeneous elements $x, y \in \mathcal{H}(A)$.
\item $(A, \bullet, \alpha)$ is also a Hom-alternative superalgebra with
$$x\bullet y=R(x)\cdot y+x\cdot R(y).$$
\end{enumerate}
\end{corollary}

\begin{proposition}
 Let $(V, \prec, \succ, \beta)$ be a bimodule over the Hom-alternative superalgebra $(A, \cdot, \alpha)$ and $R: A\rightarrow A$ be a Rota-Baxter
operator of weight $0$ on $A$. Then, $(V, \triangleleft, \triangleright, \beta)$, where
$$v\triangleleft x=v\prec R(x)\quad\mbox{and}\quad x\triangleright y=R(x)\succ v,$$
is a bimodule over $(A, \bullet, \alpha)$.
\end{proposition}
\begin{proof}
 For any homogeneous elements $x, y\in A$ and $v\in V$,
  \begin{align*}
(v\triangleleft x)\triangleleft \alpha(y)&-\beta(v)\triangleleft(x\bullet y)
\\
&=(v\prec R(x))\prec R(\alpha(y))-\beta(v)\prec R(R(x)\cdot y+x\cdot R(y)) \\
&= (v\prec R(x))\prec \alpha(R(y))-\beta(v)\prec (R(x)\cdot R(y)) \\
& =(-1)^{|x||y|}\Big( (R(x)\succ v)\prec\alpha(y)-\alpha(x)\succ(v\prec y)\Big) \\
& \stackrel{\eqref{abm1}}{=}(-1)^{|x||y|}\Big((x\triangleright v)\triangleleft\alpha(y)-\alpha(x)\triangleright(v\triangleleft y)\Big).
  \end{align*}
The other identities are proved similarly.
\end{proof}

\begin{theorem}
  Let $(A, \prec, \succ, \alpha)$ be a Hom-prealternative superalgebra, and let $\beta : A\rightarrow A$ be an even Hom-prealternative superalgebra endomorphism. Then
  $A_\beta=(A, \prec_\beta=\beta\circ\prec, \succ_\beta=\beta\circ\succ, \beta\alpha)$ is a Hom-prealternative superalgebra.
Moreover, suppose that $(A', \prec', \succ')$ is another prealternative superalgebra and $\alpha': A\rightarrow A'$ be a prealternative
superalgebra endomorphism. If $f : A\rightarrow A'$ is a prealternative superalgebra morphism that satisfies $f\circ\beta=\alpha'\circ f$, then
$$f :  (A, \prec_\beta=\beta\circ\prec, \succ_\beta=\beta\circ\succ, \beta\alpha)\rightarrow
(A', \prec_{\alpha'}'=\alpha'\circ\prec', \succ_{\alpha'}'=\alpha'\circ\succ', \alpha')$$
is a morphism of Hom-prealternative superalgebras.
\end{theorem}
\begin{proof}
For all $x, y, z\in\mathcal{H}(A)$,
\begin{align*}
&(x\succ_\beta y)\prec_\beta\beta\alpha(z)-\beta\alpha(x)\succ_\beta(y\prec_\beta z)  \\
& \qquad = \beta((\beta(x)\succ\beta(y)))\prec\beta(\alpha(z))-\beta(\alpha(x))\succ\beta[(\beta(y)\prec\beta(z))
\\
& \qquad=\beta^2 \Big((x\succ y)\prec\alpha(z)-\alpha(x)\succ(y\prec z)\Big) \\
& \qquad=(-1)^{|x||y|}\beta^2\Big(\alpha(y)\prec(x\bullet z)-(-1)^{|x||y|}(y\prec x)\prec\alpha(z)\Big) \\
& \qquad=(-1)^{|x||y|}\beta\Big(\beta\alpha(y)\prec\beta(x\bullet z)-(-1)^{|x||y|}\beta(y\prec x)\prec\beta\alpha(z)\Big) \\
& \qquad=(-1)^{|x||y|}\beta\Big(\beta\alpha(y)\prec(x\bullet_\beta z)-(-1)^{|x||y|}(y\prec_\beta x)\prec\beta\alpha(z)\Big) \\
& \qquad=(-1)^{|x||y|}\Big(\beta^2\alpha(y)\prec\beta(x\bullet_\beta z)-(-1)^{|x||y|}\beta(y\prec_\beta x)\prec\beta^2\alpha(z)\Big) \\
& \qquad=(-1)^{|x||y|}\beta\alpha(y)\prec_\beta(x\bullet_\beta z)-(-1)^{|x||y|}(y\prec_\beta x)\prec_\beta\beta\alpha(z).
  \end{align*}
The other axioms are proved similarly.
For the second part,
\begin{eqnarray*}
f\circ\prec_\alpha\!\!\!\!&=&\!\!\!\!f\circ\alpha\circ\prec=\alpha'\circ f\circ\prec
=\alpha'\circ\prec'\circ(f\otimes f)=\prec'_{\alpha'}\circ(f\otimes f).
\end{eqnarray*}
Analogues equalities hold for $\succ_\alpha$ and $\succ'_{\alpha'}$.
\end{proof}
Taking $\beta=\alpha^{2^n-1}$ yields the following result.
\begin{corollary}
Let $(A, \prec, \succ, \alpha)$ be a multiplicative Hom-prealternative superalgebra. Then,
\begin{enumerate}[label=\upshape{(\roman*)},left=2pt]
\item For $n\geq 0$,
$A^n=(A, \prec^{(n)}=\alpha^{2^n-1}\circ\prec, \succ^{(n)}=\alpha^{2^n-1}\circ\succ, \alpha^{2^n})$
is a multiplicative Hom-prealternative superalgebra, called the $n\mathrm{th}$ derived multiplicative Hom-prealternative superalgebra.
\item For $n\geq0$,
$A^n=(A, \prec^{(n)}=\alpha^{2^n-1}\circ(\prec+\succ), \alpha^{2^n})$
is a multiplicative Hom-alternative superalgebra, called the $n\mathrm{th}$ derived multiplicative Hom-alternative superalgebra.
\end{enumerate}
\end{corollary}

\subsection{Bimodules of Hom-prealternative superalgebras}
\begin{definition}
 Let $(A, \prec, \succ, \alpha)$ be a Hom-prealternative superalgebra.
An $A$-bimodule is a supervector space $V$ with an even linear map $\beta : V\rightarrow V$ and four even linear maps
 \begin{eqnarray*}
 L_\succ: A&\rightarrow&gl(V)\qquad\qquad \quad\qquad L_\prec : A\rightarrow gl(V) \\
x&\mapsto&L_\succ(x)(v)=x\succ v, \hspace{1,4cm} x\mapsto L_\prec(x)(v)= x\prec v, \\
R_\succ: A&\rightarrow& gl(V) \qquad\qquad\quad\qquad R_\prec : A\rightarrow gl(V) \\
x&\mapsto& R_\succ(x)(v)= v\succ x, \hspace{1,4cm} x\mapsto R_\prec(x)(v)= v\prec x,
 \end{eqnarray*}
satisfying the following relations:
\begin{align}
&\begin{array}{l}
L_\succ(x\bullet y+(-1)^{|x||y|}y\bullet x)\beta(v)=\\
\qquad L_\succ(\alpha(x))L_\succ(y)+(-1)^{|x||y|}L_\succ(\alpha(y))L_\succ(x),
\end{array} \label{pabm1} \\ &
\begin{array}{l}
R_\succ(\alpha(y))(L_\bullet(x)+(-1)^{|x||v|}R_\bullet(x))v= \\
\qquad L_\succ(\alpha(x))R_\succ(y)v+(-1)^{|x||v|}R_\succ(x\succ y)\beta(v),
\end{array} \\
&
\begin{array}{l}
R_\prec(\alpha(y))L_\succ(x)+(-1)^{|x||v|}R_\prec(\alpha(y))R_\prec(x)= \\
\qquad L_\succ(\alpha(x))R_\prec(y)+(-1)^{|x||v|}R_\prec(x\circ y)\beta(v),
\end{array} \\
&
\begin{array}{l}
R_\prec(\alpha(y))R_\succ(x)v+(-1)^{|x||v|}R_\succ(\alpha(y))L_\prec(x)v= \\
\qquad L_\prec(\alpha(x))R_\bullet(y)v+(-1)^{|x||v|}R_\succ(x\bullet y)\beta(v),
\end{array} \\
&
\begin{array}{l}
L_\prec(y\prec x)\beta(v)+(-1)^{|x||y|}L_\prec(x\succ y)\beta(v)= \\
\qquad L_\prec(\alpha(y))L_\bullet(x)v+(-1)^{|x||y|}L_\succ(\alpha(y))L_\succ(x)v,
\end{array} \\
&
\begin{array}{l}
R_\prec(\alpha(x))L_\succ(y)+(-1)^{|x||v|}L_\succ(y\succ x)\beta(v)= \\
\qquad L_\succ(y)R_\prec(x)v+(-1)^{|x||v|}L_\succ(\alpha(y))L_\succ(x)v,
\end{array} \\
&
\begin{array}{l}
R_\prec(\alpha(x))R_\succ(y)v+(-1)^{|x||y|}R_\succ(\alpha(y))R_\bullet(x)v= \\
\qquad R_\succ(y\prec x)\beta(v)+(-1)^{|x||y|}R_\succ(x\succ y)\beta(v),
\end{array} \\
&
\begin{array}{l}
L_\prec(y\succ x)\beta(v)+(-1)^{|x||v|}R_\succ(\alpha(x))L_\bullet(y)v= \\
\qquad L_\succ(\alpha(y))L_\prec(x)v+(-1)^{|x||v|}L_\succ(\alpha(y))R_\succ(y)v,
\end{array} \\
&
\begin{array}{l}
R_\prec(\alpha(x))R_\prec(y)v+(-1)^{|x||y|}R_\prec(\alpha(y))R_\prec(x)v= \\
\qquad R_\prec(x\bullet y+(-1)^{|x||y|}y\bullet x)\beta(v),
\end{array}
\\
&
\begin{array}{l}
R_\prec(\alpha(y))L_\prec(x)+(-1)^{|y||v|}L_\prec(x\prec y)\beta(v)= \\
\qquad L_\prec(\alpha(x))(R_\bullet(y)+(-1)^{|y||v|}L_\bullet(y))v,
\end{array}
\label{pabm10}
\end{align}
where $gl(V)$ is the set of even linear maps of $V$ onto $V$, $\bullet=\prec+\succ$ and
$$x\bullet y=x\prec y+x\succ y,\quad  L_\bullet=L_\succ+L_\prec, \quad  R_\bullet=R_\succ+R_\prec$$
for any homogeneous $x, y, v$.
\end{definition}
\begin{remark}
Axioms \eqref{pabm1}-\eqref{pabm10} are respectively equivalent to
\begin{align}
& \begin{array}{l}
(x\bullet y+(-1)^{|x||y|}y\bullet x)\succ\beta(v)=\\
\qquad \alpha(x)\succ(y\succ v)+(-1)^{|x||y|}\alpha(y)\succ(x\succ v),
\end{array}
\label{pbm1}\\
& \begin{array}{l}
(x\bullet v+(-1)^{|x||v|}v\bullet x)\succ\alpha(y) = \\
\qquad \alpha(x)\succ(v\succ y)-(-1)^{|x||v|}\beta(v)\succ(x\succ y),
\end{array}
\label{pbm2}\\
& \begin{array}{l}
(v\prec x)\prec\alpha(y)+(-1)^{|x||v|}(x\succ v)\prec\alpha(y) = \\
\qquad \beta(v)\prec(x\bullet y)+(-1)^{|x||v|} \alpha(x)\succ(v\prec y),
\end{array}
\label{pbm3}\\
& \begin{array}{l}
(x\prec v)\prec\alpha(y)+(-1)^{|x||v|}(v\succ x)\prec\alpha(y) = \\
\qquad \alpha(x)\prec(v\bullet y)+(-1)^{|x||v|} \beta(v)\succ(x\bullet y),
\end{array}
\label{pbm4}\\
& \begin{array}{l}
(y\prec x)\prec\beta(v)+(-1)^{|x||y|}(x\succ y)\prec\beta(v)= \\
\qquad \alpha(y)\prec(x\bullet v)+(-1)^{|x||y|}\alpha(x)\succ(y\prec v),
\end{array} \label{pbm5}\\
& \begin{array}{l}
(y\succ v)\prec\alpha(x)+(-1)^{|x||v|}(y\bullet x)\succ\beta(v)= \\
\qquad \alpha(y)\succ(v\prec x)+(-1)^{|x||v|}\alpha(y)\succ(x\succ v),
\end{array} \label{pbm6}\\
& \begin{array}{l}
(v\succ)\prec\alpha(x)+(-1)^{|x||y|}(v\bullet x)\succ\alpha(y)= \\
\qquad \beta(v)\succ(y\prec x)+(-1)^{|x||y|}\beta(v)\succ(x\succ y),
\end{array} \label{pbm7}\\
& \begin{array}{l}
(y\succ x)\prec\beta(v)+(-1)^{|x||v|}(y\bullet v)\succ\alpha(x)= \\
\qquad \alpha(y)\succ(x\prec v)+(-1)^{|x||v|}\alpha(y)\succ(v\succ x),
\end{array} \label{pbm8}\\
& \begin{array}{l}
(v\prec x)\prec\alpha(y)+(-1)^{|x||y|}(v\prec y)\prec\alpha(x)= \\
\qquad \beta(v)\prec(x\bullet y+(-1)^{|x||y|}y\bullet x),
\end{array}
\label{pbm9}\\
& \begin{array}{l}
(x\prec v)\prec\alpha(y)+(-1)^{|y||v|}(x\prec y)\prec\beta(v)= \\
\qquad \alpha(x)\prec(v\bullet y+(-1)^{|y||v|}y\bullet v).
\end{array}
\label{pbm10}
\end{align}
\end{remark}
\begin{proposition}
  Let $(A, \prec, \succ, \alpha)$ be a Hom-prealternative superalgebra. \\
  Then $(A, l_\succ, r_\prec, \alpha)$ is a bimodule of the associated
Hom-alternative superalgebra $Alt(A)=(A, \bullet, \alpha)$.
\end{proposition}
\begin{proposition}
 Let $(V, \prec, \succ, \beta)$ be a bimodule over the Hom-alternative superalgebra $(A, \bullet, \alpha)$ and $R: A\rightarrow A$ be a Rota-Baxter
operator on $A$. Then $(V, 0, \triangleright, 0, \triangleleft, \beta)$,
with $x\triangleright v=R(x)\succ v$ and $v\triangleleft x=v\prec R(x)$,
is a bimodule over the Hom-prealternative superalgebra $A_R=(A, \prec, \succ, \alpha)$.
\end{proposition}
\begin{proof}
For any homogeneous elements $x, y\in A$ and $v\in V$,
 \begin{align*}
&(v\triangleleft x)\triangleleft\alpha(y)+(x\triangleright v)\triangleleft\alpha(y) \\
& \qquad =(v\prec R(x))\prec R(\alpha(y))+(R(x)\succ v)\prec R(\alpha(y))\\
& \qquad =(v\prec R(x))\prec \alpha(R(y))+(R(x)\succ v)\prec \alpha(R(y))\\
& \qquad =\beta(v)\prec(R(x)\cdot R(y))+\alpha(R(x))\succ(v\prec R(y))\\
& \qquad =\beta(v)\prec R(R(x)\cdot y+ x\cdot R(y))+R(\alpha(x))\succ(v\prec R(y))\\
& \qquad =\beta(v)\triangleleft (R(x)\cdot y+ x\cdot R(y))+\alpha(x)\triangleright(v\triangleleft y)\\
& \qquad =\beta(v)\triangleleft (x\succ y+ x\prec y)+\alpha(x)\triangleright(v\triangleleft y)\\
& \qquad =\beta(v)\triangleleft (x\cdot y)+\alpha(x)\triangleright(v\triangleleft y).
 \end{align*}
The other axioms are proved in the same way.
\end{proof}

\begin{theorem} \label{thm:bimodHomaltsup}
 Let $(V, L_\prec, R_\prec, L_\succ, R_\succ, \beta)$ be a bimodule over the Hom-pre\-alternative superalgebra
$(A, \prec, \succ, \alpha)$ and $Alt(A)=(A, \bullet, \alpha)$ the associated Hom-alternative superalgebra. Then
\begin{enumerate}[label=\upshape{(\roman*)},left=2pt]
 \item \label{i:thm:bimodHomaltsup}
$(V, L_\succ, R_\prec, \beta)$ is a bimodule over $Alt(A)$.
\item \label{ii:thm:bimodHomaltsup}
$(V, L_\bullet=L_\prec+L_\succ, R_\bullet=R_\prec+R_\succ, \beta)$ is a bimodule over $Alt(A)$.
\item \label{iii:thm:bimodHomaltsup}
If $(V, L, R, \beta)$ is a bimodule of $Alt(A)$, then $(V, 0, R, L, 0, \beta)$ is a bimodule over $(A, \prec, \succ, \alpha)$.
\end{enumerate}
\end{theorem}
\begin{proof}
For any homogeneous elements $x, y\in A$ and $v\in V$,
\begin{enumerate}[label=\upshape{\roman*)},left=0pt]
\item The statement \ref{i:thm:bimodHomaltsup} follows from axioms \eqref{pbm1}, \eqref{pbm3}, \eqref{pbm6} and \eqref{pbm9}.
\item For \ref{ii:thm:bimodHomaltsup}, the axiom \eqref{abm3} is verified as follows,
\begin{align*}
&(x\bullet y)\bullet\beta(v)+ (-1)^{|x||y|} (y\bullet x)\bullet\beta(v) \\
& \quad-\alpha(x)\bullet(y\bullet v)-(-1)^{|x||y|}\alpha(y)\bullet(x\bullet v)\\
&=(x\prec y)\prec\beta(v)+(x\succ y)\prec\beta(v)+(x\bullet y)\succ\beta(v)\\
&\quad +(-1)^{|x||y|}(y\prec x)\prec\beta(v) +(-1)^{|x||y|}(y\succ x)\prec\beta(v)\\
& \hspace{6cm} +(-1)^{|x||y|}(x\bullet y)\succ\beta(v)\\
& \quad -\alpha(x)\prec(y\bullet v)-\alpha(x)\succ(y\prec v)-\alpha(x)\succ(y\succ v)\\
& \quad -(-1)^{|x||y|}\alpha(y)\prec(x\bullet v)-(-1)^{|x||y|}\alpha(y)\succ(x\prec v)\\
& \hspace{6cm} -(-1)^{|x||y|}\alpha(y)\succ(x\succ v).
\end{align*}
The left hand side vanishes by axioms \eqref{pbm1} and \eqref{pbm5}. The other axioms are verified analogously: axiom \eqref{abm4} comes from axioms \eqref{pbm7} and \eqref{pbm9}; axiom \eqref{abm2} comes from axioms \eqref{pbm6}, \eqref{pbm8} and \eqref{pbm10}; axiom \eqref{abm1} comes from axioms \eqref{pbm2}, \eqref{pbm3} and \eqref{pbm4}.
\item It suffices to take $R_\succ=0$ and $L_\prec=0$.
\qedhere
\end{enumerate}
\end{proof}
\begin{theorem}
  Let $(V, L_\prec, R_\prec, L_\succ, R_\succ, \beta)$ be a bimodule over the multiplicative Hom-prealternative superalgebra
$(A, \prec, \succ, \alpha)$, and let $Alt(A)=(A, \bullet, \alpha)$ be the associated Hom-alternative superalgebra.
Then both
$(V, L_\succ^\alpha, R_\prec^\alpha, \beta)$ and  $(V, L_\bullet^\alpha=L_\prec^\alpha+L_\succ^\alpha, R_\bullet^\alpha=R_\prec^\alpha+R_\succ^\alpha, \beta)$ are bimodules over $Alt({A})$, where
\begin{align*}
& L_\prec^\alpha=L_\prec\circ(\alpha^2\otimes Id), \quad L_\succ^\alpha=L_\succ\circ(\alpha^2\otimes Id), \\
& R_\prec^\alpha=R_\prec\circ(\alpha^2\otimes Id), \quad R_\succ^\alpha=R_\succ\circ(\alpha^2\otimes Id).
\end{align*}
\end{theorem}
\begin{proof}
We only prove \eqref{abm3} in detail, as the other axioms are verified similarly.
Putting $\succ_\alpha=L_\succ^\alpha$, for any homogeneous elements $x, y\in A$ and $v\in V$,
\begin{align*}
&(x\bullet y+ (-1)^{|x||y|}y\bullet x)\succ_\alpha\beta(v)=\alpha^2(x\bullet y+(-1)^{|x||y|}y\bullet x)\succ\beta(v)
\\
& \stackrel{\eqref{BSprealt:multiplicativityalpha}}{=} 
 (\alpha^2(x)\bullet \alpha^2(y)+(-1)^{|x||y|}\alpha^2(y)\bullet \alpha^2(x))\succ\beta(v)  \\
& \stackrel{\eqref{pbm1}}{=}\alpha^3(x)\succ(\alpha^2(y)\succ v)+(-1)^{|x||y|}\alpha^3(y)\succ(\alpha^2(x)\succ v) \\
& = \alpha(x)\succ_\alpha(y\succ_\alpha v)+(-1)^{|x||y|}\alpha(y)\succ_\alpha(x\succ_\alpha v).
\qedflrght \end{align*}
\phantom{\qedhere}
\end{proof}

\section*{Acknowledgments}
Dr. Ibrahima Bakayoko is grateful to the research environment in Mathematics and Applied Mathematics MAM, Division of Mathematics and Physics of the School of Education, Culture and Communication at M{\"a}lardalen University for hospitality and an excellent and inspiring environment for research and research education and cooperation in Mathematics during his visit in Autumn of 2019, which contributed towards expanding research and research education capacity and cooperation development in Africa and impact of the programs in Mathematics between Sweden and countries in Africa supported by Swedish International Development Agency (Sida) and International Program in Mathematical Sciences (IPMS).
Partial support from Swedish Royal Academy of Sciences is also gratefully acknowledged.


\end{document}